\documentclass[a4paper]{article}

\usepackage{graphicx}
\usepackage{xcolor}
\usepackage[utf8]{inputenc}
\usepackage[T1]{fontenc}
\usepackage{amsmath}
\usepackage{amssymb}
\usepackage{amsthm}
\usepackage{multirow}
\usepackage{subfigure}
\usepackage{a4}

\usepackage[normalem]{ulem}

\definecolor{GanzDklMgta}{rgb}{0.4,0.12,0.6}
\definecolor{DklCyan}{rgb}{0.12,0.47,0.58}

\newtheorem{ntn}{Notation}

\newtheorem{lemma}{Lemma}
\newtheorem{theorem}{Theorem}
\newtheorem{corollary}{Corollary}

\newtheorem{example}{Example}
\newcommand{\R}{{\rm I\hspace{-0.15em}R}}

\newcommand{\N}{{\rm I\hspace{-0.15em}N}}
\newcommand{\cU}{\mathcal{U}}

\newcommand{\F}{\mathcal{F}}
\newcommand{\X}{\mathcal{X}}
\newcommand{\G}{\mathcal{G}}
\newcommand{\SR}{{\rm S \hspace{-0.15em} R}}

\newcommand{\Peps}{(Rec($\varepsilon$))}
\newcommand{\Pdelta}{(Rec$^{{\rm class}}$($\delta$))}
\newcommand{\Geps}{\G_{\varepsilon}}
\newcommand{\reps}{r_{\varepsilon}}

\newcommand{\f}{{\bf f}}
\newcommand{\br}{{\bf r}}

\usepackage{authblk}
\usepackage{chngpage}
\begin{document}

\title{A biobjective approach to robustness based on location planning \thanks{Partially supported by grants SCHO 1140/3-2 within the DFG programme {\it Algorithm Engineering}, and grant MTM2012-36163-C06-03.}}

\author[1]{Emilio Carrizosa}
\author[2]{Marc Goerigk\thanks{Corresponding author. Email: m.goerigk@lancaster.ac.uk}}
\author[3]{Anita Sch\"obel}

\date{}
\affil[1]{{\small Universidad de Sevilla, Spain}}
\affil[2]{{\small University of Lancaster, United Kingdom}}
\affil[3]{{\small Georg-August Universit\"at G\"ottingen, Germany}}

\maketitle

\begin{abstract}
Finding robust solutions of an optimization problem is an
important issue in practice, and various concepts on how to define the
robustness of a solution have been suggested. The idea of recoverable robustness
requires that a solution can be recovered to a feasible 
one as soon as the realized scenario becomes known.
The usual approach in the literature is to minimize the objective function value of the recovered solution
in the nominal or in the worst case.

As the recovery itself is also costly, there is a trade-off 
between the recovery costs and the solution value obtained;
we study both, the recovery costs and
the solution value in the worst case in a biobjective setting.

To this end, we assume that the recovery costs
can be described by a metric. We demonstrate that this leads to a location planning problem, bringing together two fields of research which have been considered separate so far.

We show how weakly Pareto efficient solutions to
this biobjective problem can be computed by minimizing the recovery costs for
a fixed worst-case objective function value and present approaches for the
case of linear and quasiconvex problems for finite uncertainty sets. 
We furthermore
derive cases in which the size of the uncertainty set can be reduced without
changing the set of Pareto efficient solutions.
\end{abstract}


\section{Introduction}

Robust optimization is a popular paradigm to handle optimization problems contaminated with uncertain data, see, e.g., \cite{RObook,Aissi2009} and references therein.

Starting from conservative robustness models requiring that the robust solution
is feasible for any of the possible scenarios, new concepts have been developed,
see \cite{GoeSchoe13-AE} for a recent survey.
These concepts allow to relax this conservatism and to control the {\it price of robustness}, i.e.,
the loss of objective function value
one has to pay in order to obtain a robust solution, see \cite{BertSim04}.
In many real-world problems these new robustness concepts have been successfully
applied.

Motivated by two-stage stochastic programs, one class of such new models includes the so called 
{\it recoverable robustness} introduced
in \cite{LLMS09,CDDFN07}
and independently also used in \cite{Savel09}. Recoverable robustness is a two-stage approach that
does not require the robust solution to be feasible for all scenarios. Instead, a recoverable-robust
solution comes together with a \emph{recovery strategy} which is able to 
adapt the solution to make it feasible for every scenario.
Such a recovery strategy can be obtained by modifying the values of the
solution or by allowing another resource or spending additional budget, as soon as it becomes known
which scenario occurs. Unfortunately, a recoverable-robust
solution can only be determined efficiently for simple problems with special assumptions on the
uncertainties and on the recovery algorithms (see \cite{sebastiandiss}),
and the recoverable-robust counterpart is known to be NP-hard even in simple cases \cite{ARRIVAL-TR-0172}.

\smallskip

\textit{Our contributions.}
In this paper we analyze the two main goals in recoverable robustness: Obtaining a good objective
function value in the worst case while minimizing the recovery costs. We consider the $\varepsilon$-constrained version as a geometric problem, which allows to interpret robustness as a location planning problem,
and derive results on Pareto efficient solutions and how to compute them.

\smallskip

\textit{Overview.}
The remainder of the paper is structured as follows.

In the next section we sketch classic and
more recent robustness concepts before we
introduce the biobjective version of recoverable robustness
in Section~\ref{sec-recrob}.
We then analyze how to solve the scalarization of the recoverable-robust counterpart in Section~\ref{sec:solve},
and consider reduction approaches in Section~\ref{sec-infinite}.
After discussing numerical experiments in Section~\ref{sec:exp}, we conclude with a summary of results and an outlook to further research in Section~\ref{sec-conclusion}.

\section{Robustness concepts}
\label{sec-intro}

\subsection{Uncertain optimization problems}

We consider optimization problems that can be written in the form
\begin{eqnarray*}
\label{nominal-1}
\mbox{(P)} \hspace{1cm} & {\rm minimize} \ f(x) \\
& \mbox{s.t. }   F(x) \leq 0 \\
& x \in \X,
\end{eqnarray*}
where $\X \subseteq \R^n$ is a closed set, 
$F:\R^n \to \R^m$ describes the $m$ constraints
and $f:\R^n \to \R$ is the objective function to be minimized. We assume $f$ and 
$F$ to be continuous.
In practice, the constraints and the objective
may both depend on parameters which are in many cases not exactly known. 
In order to accommodate such uncertainties,
the following class of problems is considered instead of $\rm (P)$.

\begin{ntn} 
An \emph{uncertain optimization problem} is given as a parameterized family of optimization problems
\begin{eqnarray*}
\label{nominal-2}
\mbox{P($\xi$)} \hspace{1cm} & {\rm minimize} \  f(x,\xi) \\
& \mbox{s.t. }   F(x,\xi) \leq 0\\
& x \in \X,
\end{eqnarray*}
where $F(\cdot, \xi):\R^n \to \R^m$ and $f(\cdot, \xi):\R^n \to \R$
are continuous functions for any fixed $\xi\in\cU$, $\cU \subseteq \R^M$ being the 
\emph{uncertainty set} which contains all possible scenarios $\xi \in \R^M$ which may occur (see also \cite{RObook}). 
\end{ntn}

A scenario $\xi \in \cU$ fixes the parameters
of $f$ and $F$. It is often known that all scenarios that may occur 
lie within a given uncertainty set $\cU$, however, it
is not known beforehand which of the scenarios $\xi \in \cU$ will be realized.
We assume that $\cU$ is a closed set in $\R^M$ containing at least two
elements (otherwise, no uncertainty would affect the problem).
Contrary to the setting of stochastic optimization problems, we do not assume
a probability distribution over the uncertainty set to be known.
\medskip

The set $\X$ contains constraints which do not depend on the uncertain parameter $\xi$.
These may be technological or physical constraints on the variables
(e.g., some variables represent non-negative magnitudes, or there are
precedence constraints between two events), or may refer to modeling
constraints (e.g., some variables are Boolean, and thus they can only take
the values 0 and 1).
\medskip

In short, the uncertain optimization problem corresponding to P($\xi$) is denoted as
\begin{equation}
\label{uncertain}
 (\mbox{P($\xi$)}, \xi \in \cU).
\end{equation}
We denote
\[ \F(\xi) =  \{ x \in \X: F(x,\xi) \leq 0\} \]
as the feasible set of scenario $\xi \in \cU$ and
\[ f^*(\xi)=\min\{f(x,\xi): F(x,\xi) \leq 0, x \in \X \} \]
as the optimal objective function value for scenario $\xi$ 
(which might be $\infty$ if it does not exist).
Note that $\F(\xi)$ is closed in $\R^n$, as we assumed 
$\X$ to be closed, and $F(\cdot,\xi)$ to be continuous.
In the following we demonstrate the usage of $\xi \in \R^M$ 
for the case of linear optimization.
In the simplest case,
$\xi$  \textit{coincides with the uncertain
  parameters} of the given optimization problem.

\medskip

\begin{example}
\label{ex-linear}
Consider a linear program ${\rm minimize} \{c^tx: Ax \leq b, x \in \R^n\}$ with
a coefficient matrix $A \in \R^{m,n}$, a right-hand side vector $b \in
\R^m$ and a cost vector $c \in \R^n$. If $A,b$, and $c$ are treated as uncertain 
parameters, we write
\begin{eqnarray*}
\label{nominal-3}
\mbox{P($A,b,c$)} \hspace{1cm} & {\rm minimize} \ f(x,(A,b,c))=c^tx \\
& \mbox{s.t. }   F(x,(A,b,c))=Ax-b \leq 0 \\
& x \in \X = \R^n,
\end{eqnarray*}
i.e., $\xi=(A,b,c) \in \R^M$ with $M=n \cdot m + n + m$
\end{example}

However, in (\ref{uncertain}) we allow a more general setting, namely
that the unknown parameters $A,b,c$ may depend on (other) uncertain
parameters $\xi \in \R^M$. For example, there might
be $M=1$ parameter $\xi \in \R$ which determines all values of
$A,b,c$. As an example imagine that the temperature determines the
properties of different materials. In such a case we would have
\begin{eqnarray*}
f(x,\xi):\R^n \times \R \to \R, \mbox{ and } \\
F(x,\xi):\R^n \times \R \to \R^m,
\end{eqnarray*}
where $f(x,\xi)=c(\xi)^t x$ and $F(x,\xi)=A(\xi)x-b(\xi)$. 
\medskip

We now summarize several concepts to handle uncertain optimization problems.

\subsection{Strict robustness and less conservative concepts}
\label{sec:strict}

The 
first formally introduced robustness concept is 
called \textit{strict robustness} here. It
has been first mentioned by
Soyster \cite{Soyster} and then formalized and analyzed by Ben-Tal, El Ghaoui,
and Nemirovski in numerous publications, see
\cite{RObook} for an extensive collection of results.
A solution $x \in \X$ to the uncertain problem (\ref{uncertain})
is called \textit{strictly robust} if it is
feasible for all scenarios in $\cU$, i.e., if
$F(x,\xi) \leq 0$ for all $\xi \in \cU$.
The set of strictly robust solutions with respect to the uncertainty set $\cU$ 
is denoted by $\SR(\cU) = \bigcap_{\xi\in\cU} \F(\xi)$.
The \emph{strictly robust counterpart} of (\ref{uncertain}) is given as
\begin{eqnarray*}
\label{RC-1}
\mbox{RC}(\cU) \hspace{1cm} & {\rm minimize} \ \sup_{\xi \in \cU} f(x,\xi) \label{RC-11}\\
& \mbox{s.t. } x \in \SR(\cU)
\end{eqnarray*}
The objective follows the pessimistic view of minimizing
the worst case over all scenarios.

Often the set of strictly robust solutions is empty, or all of the strictly
robust solutions lead to undesirable solutions (i.e., with considerably worse objective values than a 
nominal solution would achieve).
Recent concepts of
robustness hence try to overcome the ``over-conservative'' nature of the previous approach.
In this paper we deal with recoverable robustness which is described in the next section.
Other less conservative approaches include 
the approach of Bertsimas and Sim \cite{BertSim04}, reliability
\cite{BenTalNemi2000}, light robustness \cite{FischMona09,Scho13},
adjustable robustness \cite{BeGoGuNe2003} (which will be used in Section~\ref{sec-relations}),
and comprehensive robustness \cite{BTBN06}.
For a more detailed recent 
overview on different robustness concepts we refer to \cite{GoeSchoe13-AE}. 

\section{A biobjective approach to recoverable robustness}
\label{sec-recrob}

Our paper extends the recently published concepts of \textit{recoverable robustness}. As before, we consider a parameterized problem
\begin{eqnarray*}
\mbox{P($\xi$)} \hspace{1cm} & {\rm minimize} \ f(x,\xi) \\
& \mbox{s.t. } F(x,\xi) \leq 0\\
& x \in \X
\end{eqnarray*}
The idea of recoverable robustness (see \cite{LLMS09})
is to allow that a solution can be recovered to a feasible one 
for every possible scenario. There, a solution $x \in \X$ is called
\textit{recoverable-robust} if there is a function $y: \cU \to \X$ such that
for any possible scenario $\xi \in \cU$, the solution $y(\xi)\in \F(\xi)$
is not too different from the original solution $x$. This includes on the
one hand the costs for changing the solution $x$ to the solution $y(\xi)$, 
and on the other hand the objective function value of $y(\xi)$ compared to 
the objective function value of $x$.
The solution $y(\xi)$ is called the \emph{recovery solution} for scenario $\xi$.

Examples include recoverable-robust models for
linear programming \cite{sebastiandiss}, shunting \cite{CDDFN07},
timetabling \cite{robustness-overview08}, platforming \cite{rrplatforming},
the empty repositioning problem \cite{Savel09}, railway rolling stock planning \cite{Kroon10}
and the knapsack problem \cite{BKK11}.
An extensive investigation can be found in \cite{sebastiandiss}.
Note that the model has the drawback that even for simple optimization problems an optimal
recoverable-robust solution is usually hard to determine.

\subsection{Model formulation}

Various goals may be followed when computing a recoverable-robust solution: On the one
hand, the new solution should be recoverable to a good solution $y(\xi) \in \F(\xi)$ for every
scenario $\xi \in \cU$. 
On the other hand, also the costs of the recovery are important: A new solution has to be
implemented, and if $x$ differs too much from $y(\xi)$ this might be too costly. 
We assume that the recovery costs can be measured by a metric 
$d:\R^n \times \R^n \to \R$.
An example for metric recovery costs can be found, e.g., for shunting in \cite{Kroon10};
recovery costs defined by norms are also used frequently,
e.g., in timetabling \cite{LLMS09}, in recoverable-robust
linear programming \cite{sebastiandiss}, or in vehicle scheduling problems \cite{Goerigk201566}.
\medskip

Our biobjective model for recoverable robustness can be formulated as follows:
\begin{eqnarray*}
\mbox{(Rec)} \hspace{5mm} & {\rm minimize} \ 
\left( {\bf f}(y), {\bf r}(x,y) \right)=
\left( \sup_{\xi \in \cU} f(y(\xi),\xi), \sup_{\xi \in \cU} d(x,y(\xi)) \right)\\
& \mbox{s.t. } F(y(\xi),\xi) \leq 0\ \ \mbox{ for all } \xi \in \cU\\
& x \in \X, y:\cU \to \X
\end{eqnarray*}
We look for a recoverable robust solution $x$ together with  a recovery solution $y(\xi) \in \F(\xi)$ for every scenario $\xi \in \cU$. Note that if $\cU$ is infinite, (Rec)
is not a finite-dimensional problem. In the objective function we consider
\begin{itemize}
\item the quality $f(y(\xi),\xi)$ of the recovery solutions, which will finally be implemented, in the worst case, and
\item the costs of the recovery $d(x,y(\xi))$, i.e., changing $x$ to $y(\xi)$, again in the worst case.
\end{itemize}

As usual in multi-criteria optimization we are interested in finding Pareto efficient
solutions to this problem. Recall that a solution $(x \in \X, y:\cU \to \X)$ is weakly
Pareto efficient
if there does not exist another solution $x' \in \X, y':\cU \to \X$ such that
\begin{eqnarray*}
\sup_{\xi \in \cU} f(y'(\xi),\xi) & < & \sup_{\xi \in \cU} f(y(\xi),\xi) \mbox{ and }\\
\sup_{\xi \in \cU} d(x',y'(\xi))  & < & \sup_{\xi \in \cU} d(x,y(\xi)).
\end{eqnarray*}
If there does not even exist a solution $x' \in X, y':\cU \to \X$ for which one of the two inequalities holds with equality, then $(x, y)$ is called Pareto efficient. 
\medskip

\begin{ntn}
We call $x$ \emph{recoverable-robust} for (Rec) if there exists $y:\cU \to \X$ 
such that $(x,y)$ is Pareto efficient for (Rec).
$y(\xi)$ is called the recovery solution for scenario $\xi$.
\end{ntn}

We are interested in finding recoverable-robust 
solutions $x$. Note that (Rec) depends on the uncertainty
set $\cU$. This dependence is studied in Section~\ref{sec-infinite}.
\medskip

In (Rec), the worst-case objective $\f$ does not depend on $x$. This is because we assume 
that $x$ is always modified to the appropriate solution $y(\xi)$ when the scenario is revealed.

We remark, that even if $x$ or $y$ is fixed, the resulting problem (Rec) is still challenging.
If $x$ is given, we still have to solve a biobjective problem and choose $y(\xi)$ either with a good 
objective function value in scenario $\xi$ or with small recovery costs close to $x$. If $y$ is given,
(Rec) reduces to a single-objective problem in which a point is searched which minimizes the maximum
distance to all points $y(\xi), \xi \in \cU$.
\medskip

Our first result is negative: Pareto efficient solutions need not exist even for
a finite uncertainty set and bounded recovery costs as the following example demonstrates.

\begin{example}
Consider the uncertain program
\begin{align*}
\mbox{P($\xi$)} \hspace{1cm} \min\  &f(x,\xi)\\
\mbox{s.t. }  1 &\leq \xi x_1 x_2 \\
\xi x_1 &\geq 0\\
x_2 &\geq 0,
\end{align*}
where   $\mathcal U=\{-1,1\}$ is the uncertainty set 
and $\X=\R^2$. The feasible sets of scenario $\xi_1=-1$ and scenario $\xi_2=1$ are given by:
\begin{align*}
\mathcal F(-1) &= \{(x_1,x_2)\in\R^2 : x_1x_2\leq -1, ~x_1\leq 0, ~x_2\geq 0\},\\
\mathcal F(1) &= \{(x_1,x_2)\in\R^2 : x_1x_2\geq 1, ~x_1,x_2\geq 0\}.
\end{align*}
Both feasible sets are depicted in Figure~\ref{figexamplesets}. 
For the objective function $f(x,\xi)=|x_1|$, the problem does not have any 
Pareto efficient solution.

\begin{figure}[htbp]
\begin{center}
\includegraphics[width=0.5\textwidth]{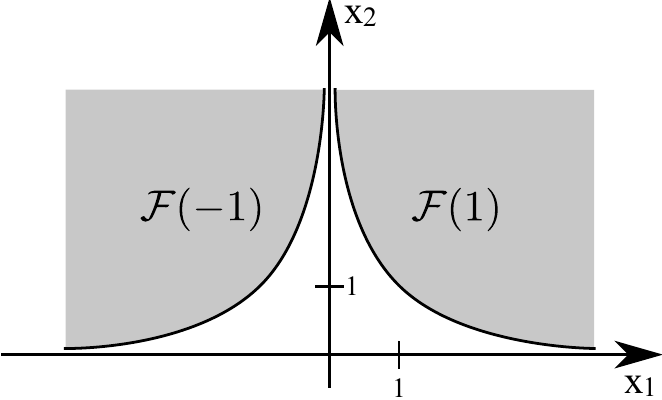}
\caption{An instance of (Rec) not having any Pareto efficient solution
for an uncertainty set $\cU$ with only two scenarios.
\label{figexamplesets}}
\end{center}
\end{figure}
\end{example}

It is known that all weakly Pareto efficient solutions
are optimal solutions of
one of the two $\varepsilon$-constraint scalarizations which
are given by bounding one of the objective functions while minimizing the other one.
\medskip

The first scalarization bounds the recovery costs
and minimizes the objective function value in the first place, i.e.,
\begin{eqnarray*}
\mbox{\Pdelta}
\hspace{1cm} & {\rm minimize} \ \sup_{\xi \in \cU} f(y(\xi),\xi) \\
\mbox{s.t. } 
& d(x,y(\xi)) \leq \delta \ \ \mbox{ for all } \xi \in \cU\\
& F(y(\xi),\xi) \leq 0\ \ \mbox{ for all } \xi \in \cU\\
&x \in \X, y:\cU \to \X.
\end{eqnarray*}
This problem has been introduced as \emph{recoverable robustness} (see \cite{LLMS09}) and
solved in several special cases, e.g., 
in \cite{selection,Goerigk201566,BKK11}.
It is hence denoted as the \emph{classic} scalarization approach.
\medskip

In our paper we look at the other scalarization in which we minimize the recovery
costs while requiring a minimal quality of the recovery solutions:
\begin{eqnarray}
\mbox{\Peps}
\hspace{1cm} & {\rm minimize} \ \sup_{\xi \in \cU} d(x,y(\xi)) \nonumber\\
\mbox{s.t. } 
& f(y(\xi),\xi) \leq \varepsilon\ \ \mbox{ for all } \xi \in \cU \label{unc1} \\
& F(y(\xi),\xi) \leq 0\ \ \mbox{ for all } \xi \in \cU \label{unc2} \\
&x \in \X, y:\cU \to \X \nonumber
\end{eqnarray}
Note that Constraints~\eqref{unc1} and \eqref{unc2} of this second scalarization do \emph{not} 
depend on $x$. To determine feasibility of \Peps, we hence check if for every $\xi \in \cU$ there exists $y(\xi)$ such that 
\[ f(y(\xi),\xi) \leq \varepsilon \mbox{ and } F(y(\xi),\xi) \leq 0, \]
i.e., if the sets
\begin{eqnarray*}
\Geps(\xi) & := & \{y\in\X : F(y,\xi) \leq 0 \mbox{ and } f(y,\xi) \leq \varepsilon\}\\
                      & = &  \F(\xi) \cap \{y\in\X : f(y,\xi) \leq \varepsilon\}
\end{eqnarray*}
are not empty for all $\xi \in \cU$. To extend some given $x$ to a feasible solution, we
choose some $y(\xi) \in \Geps(\xi)$ which is closest to $x$ w.r.t the metric $d$.
This is possible since 
$\Geps(\xi)$ is closed: we define 
\[ d(x,\Geps(\xi))=\min_{y \in \Geps(\xi)} d(x,y),\]
where the minimum exists whenever $\Geps(\xi) \not= \emptyset$. 

With $d(x,\emptyset):=\infty$, we can now rewrite 
(Rec($\varepsilon$)) to an equivalent problem in the (finitely many) $x$-variables only:
\begin{eqnarray}
\label{peps}
\mbox{(Rec'($\varepsilon$))} 
\hspace{1cm} & {\rm minimize} \  \sup_{\xi \in \cU} d(x,\Geps(\xi)) \\
& \mbox{s.t. } x \in \X \nonumber,
\end{eqnarray}
i.e., $x$ is an optimal solution to (Rec'($\varepsilon$)) if and only if $(x,y)$ with
$y(\xi) \in {\rm argmin}_{\Geps(\xi)} d(x,y)$ is optimal for (Rec($\varepsilon$)). 

\subsection{Location-based interpretation of {\Peps}}

In a classic location problem (known
as the \textit{Weber problem} or as the \textit{Fermat-Torricelli problem}, see e.g., 
\cite{DKSW02}) we have given a set of points,
called {\it existing facilities}, and we look for a new
point minimizing a measure of distance to these given points. If the distance to the 
farthest point is considered as the objective function, the problem is called 
{\it center location problem}. 
We have already seen that for given $y:\cU \to \X$, our biobjective problem (Rec) reduces to
the problem of finding a location $x$ which minimizes the maximum distance to the set
$\{y(\xi): \xi \in \cU\}$, i.e., a classic center location problem.

\medskip

We now show that also the $\varepsilon$-constrained version {\Peps} of recoverable robustness
\[ \min_{x \in \X} \max_{\xi\in\cU} d(x,\Geps(\xi)) \]
can be interpreted as 
the following location problem: The existing facilities
are not points but the sets $\Geps(\xi), \xi \in \cU$.  {\Peps} looks for
a new location in the metric space $\X$, namely a point
$x \in \X$ which minimizes the maximum distance to the given sets.
For a finite uncertainty set $\cU$, 
such location problems have been studied in \cite{brimberg2000note,brimberg2002locating}
for the center objective function and in \cite{brimberg2002minisum,nickel2003approach} for
median or ordered median objective functions.
We adapt the notation of location theory and call such a point (which then is an optimal
solution to {\Peps} a \textit{center} with respect to $\{\Geps(\xi): \xi \in \cU\}$ and the distance 
function $d$.
In our further analysis we consider {\Peps} from a location's point of view. To this end,
let us denote the objective function
of {\Peps} by
\[ \reps(x, \cU)=\sup_{\xi \in \cU} d(x,\Geps(\xi)) \]
and let us call $\reps(x, \cU)$ the (recovery) radius of $x$ with respect to $\varepsilon$ and $\cU$.
Let $\reps^*(\cU)$ denote the best possible recovery radius over $\X$ (if it exists). 
For a center location $x^*$ we then have $\reps(x^*, \cU)= \reps^*(\cU)$.

\medskip

The algorithmic advantage drawn from the connection between {\Peps} and (point) 
location problems becomes clear for specific shapes of the sets $\Geps(\xi).$
For instance, let $P(\xi)$ be given as 
\[ \min\{\|\xi-x\|: x \in \R^n\} \]
for some norm $\| \cdot \|$ and $\xi \in \cU \subseteq \R^n$. Then
the sets $\Geps(\xi)$ are scaled and
translated unit balls of the norm $\| \cdot \|$, i.e.,
\[ \Geps(\xi)=\{x \in \R^n: \|x - \xi\| \leq  \varepsilon \}. \]
In this case we obtain that
\[ d(x,\Geps(\xi))=\left\{ \begin{array}{ll}
d(x,\xi)-r & \mbox{ if } d(x,\xi) > \varepsilon \\
0                   & \mbox{ if } d(x,\xi) \leq \varepsilon
\end{array} \right. \]
and it turns out that the center of the location problem with existing (point) facilities
$\{\xi: \xi \in \cU\}$ is an optimal solution to {\Peps} and hence weakly Pareto efficient for
(Rec).

\subsection{Relation of the biobjective model to other robustness concepts}
\label{sec-relations}

We first point out the relation between (Rec) and the concept of strict
robustness of \cite{RObook}. To this end recall from Section~\ref{sec:strict} that
$\SR(\cU)=\{x\in\X : F(x,\xi) \leq 0 \mbox{ for all } \xi \in \cU\}$
is the set of strictly robust solutions and RC($\cU)$ is the strictly robust counterpart of 
$(P(\xi),\xi \in \cU)$.

\begin{lemma}
Let an uncertain problem $\rm (P(\xi), \xi \in \cU)$ be given. Then we have:
\begin{enumerate}
\item If $\bar{x}$ is an optimal solution to RC($\cU$) then $(\bar{x},\bar{y})$ with $\bar{y}(\xi):=\bar{x} \mbox{ for all } \xi \in \cU$
is a lexicographically minimal solution to (Rec) w.r.t $({\bf r}(x,y),\f(y))$.
\item Let $(\bar{x},\bar{y})$ be a lexicographically minimal solution to (Rec) w.r.t $({\bf r}(x,y),{\bf f}(y))$.
Then $\SR(\cU)\not=\emptyset$ if and only if ${\bf r}(\bar{x},\bar{y})=0$ and in this case $(\bar{x},\bar{y})$
is optimal to RC($\cU$).
\end{enumerate}
\end{lemma}

\begin{proof}
\begin{enumerate}
\item Let $\bar{x}$ be an optimal solution to RC($\cU$). Define $\bar{y}(\xi):=\bar{x}$ for all $\xi \in \cU$. Then
${\bf r}(\bar{x},\bar{y})=0$. Now assume $(\bar{x},\bar{y})$ is not lexicographically minimal. 
Then there exists $(x',y')$ with ${\bf r}(x',y')=0$ and 
${\bf f}(y') < {\bf f}(\bar{y})$. The first condition yields that $d(x',y'(\xi))=0$ for
all $\xi \in \cU$, hence $x'=y'(\xi)$ for all $\xi \in \cU$, and $x' \in \SR(\cU)$. Using $x'=y'(\xi)$, the 
second condition implies $\sup_{\xi \in \cU} f(x',\xi) < \sup_{\xi \in \cU} f(\bar{x},\xi)$, a contradiction to the
optimality of $\bar{x}$ for RC($\cU$).
\item Now let $(\bar{x},\bar{y})$ be lexicographically minimal to (Rec). 
\begin{itemize}
\item Let ${\bf r}(\bar{x},\bar{y})=0$. Then 
$0={\bf r}(\bar{x},\bar{y})=\sup_{\xi \in \cU}d(\bar{x},\bar{y}(\xi))$, 
i.e., $\bar{x}=\bar{y}(\xi)$ for all $\xi \in \cU$. Hence 
$\bar{x} \in \F(\xi)$ for all $\xi \in \cU$, i.e., $\bar{x} \in \SR(\cU)$.
\item On the other hand, if $\SR(\cU) \not= \emptyset$ there exists $x \in \F(\xi)$ for all $\xi \in \cU$.
We define $y(\xi):=x$ for all $\xi \in \cU$ and obtain ${\bf r}(x,y)=0$. Since $(\bar{x},\bar{y})$ is 
lexicographically minimal this implies ${\bf r}(\bar{x},\bar{y})=0$. 
\end{itemize}
Finally, if ${\bf r}(\bar{x},\bar{y})=0$ we already know that $\bar{x}=\bar{y}(\xi)$ for all $\xi \in \cU$ and $\bar{x} 
\in \SR(\cU)$, i.e., feasible for RC($\cU$).
The lexicographic optimality then guarantees that $\bar{x}$ is an optimal solution to RC($\cU$).
\end{enumerate}
\end{proof}

Sorting the criteria in the objective function in the other order, i.e., minimizing first
${\bf f}(y)$ and then ${\bf r}(x,y)$ is not directly related to any known robustness concept.
This lexicographically minimal solution $(x,y)$ realizes an optimal solution $y(\xi)$ in every
scenario, and among these optimal solutions minimizes the recovery costs.

\begin{lemma}
Let $(x,y)$ be a solution to (Rec) which is lexicographically minimal w.r.t 
$({\bf f}(y),{\bf r}(x,y))$. Then $f(y(\xi),\xi)=f^*(\xi)$ for all $\xi \in \cU$.
\end{lemma}

We now turn our attention to {\Peps} and show that this scalarization can be interpreted as
adjustable robustness as in \cite{BeGoGuNe2003}.
Motivated by stochastic programming, 
the variables in this concept are decomposed into two sets: The values
for the \textit{here-and-now variables} have to be found in the robust optimization algorithm while the decision about the \textit{wait-and-see variables} can wait until the actual scenario $\xi \in \cU$ becomes known.
For an uncertain problem $\rm (P(\xi), \xi \in \cU)$, recall that {\Peps} is given as
\[ \min_{x\in\X} \sup_{\xi\in\cU} d(x,\Geps(\xi)).\]
We can rewrite this problem in the following way:
\[ \min_{z,x} \left\{ z : \forall \xi\in\cU \ \exists y \in \Geps(\xi) : d(x,y) \le z \right\} \]
which has the same structure as an adjustable robust problem. As an example, for a 
problem with linear objective function $f(x,\xi)=c(\xi)^t x$,
linear constraints $F(x,\xi) = A(\xi)x - b(\xi) \le 0$, and $\Vert\cdot\Vert_1$ as recovery norm, we may write
\begin{equation}\label{adjust-program}
\min_{z',x} \left\{ \sum_{i=1}^n z'_i : \forall \xi\in\cU\ \exists y : c(\xi)^ty \leq \varepsilon,\ 
A(\xi)y \le b(\xi),\ -z' \le x - y \le z' \right\}.
\end{equation}
Note that this is a problem \emph{without} fixed recourse, such that most of the results in 
\cite{BeGoGuNe2003} are \emph{not} applicable.
However, we are still able to apply their results on using heuristic, affinely adjustable counterparts, and
Theorem 2.1 from \cite{BeGoGuNe2003}:

\begin{theorem}
Let $\rm (P(\xi), \xi \in \cU)$ be an uncertain linear optimization problem, 
and let the uncertainty be constraint-wise. Furthermore, let there be a compact set $C$ such that 
$\F(\xi) \subseteq C$ for all $\xi\in\cU$. Then, {\Peps} 
is equivalent to the following problem
\begin{equation}\label{adjust-program2}
\min_{z,x} \left\{ \sum_{i=1}^n z'_i :\exists y\ \forall \xi\in\cU: c(\xi)^ty \leq \varepsilon,\ 
A(\xi)y \le b(\xi),\ -z' \le x - y \le z' \right\}.
\end{equation}
\end{theorem}

Note that problem \eqref{adjust-program2} is a strictly robust problem, which is considerably easier to solve than problem \eqref{adjust-program}. Furthermore, 
\cite{BeGoGuNe2003} show that
there is a semidefinite program for ellipsoidal uncertainty sets 
which is equivalent to problem \eqref{adjust-program}.

Problem {\Peps} can 
also be interpreted as a strictly robust problem in $x$ (see \eqref{peps}).
However, the function $\xi \mapsto d(x,\Geps)$ has in general not much properties such that most of the
known results cannot be directly applied. Nevertheless, our geometric interpretation gives rise to the results of
the next section, in particular within the biobjective setting.

\section{Solving \Peps}
\label{sec:solve}

In this section we investigate the new scalarization {\Peps}. 
After a more general analysis of this optimization
problem in Section~\ref{sec41}, we turn our attention to the case of a finite uncertainty set
in Section~\ref{sec-finite} where we consider problems with convex and with linear
constraints.

\subsection{Analysis of \Peps}\label{sec41}
Let us now describe some general properties of problem {\Peps}.
Since $d$ is a metric we know that
\begin{equation}
\label{lower-bound}
0 \leq \reps(x,\cU)  \leq +\infty  \ \mbox{ for all } x \in \R^n,
\end{equation}
hence the optimal value of {\Peps} is bounded by zero from below, although it is 
$+\infty$ if all points $x$ have infinite radius
$\reps(x,\cU).$ This event may happen even when all
sets $\Geps(\xi)$ are non-empty. Indeed, consider, for 
instance, $\X=\R,$ $\Geps(\xi) = \{\xi\}$ for all $\xi \in \cU= \R.$
One has, however, that finiteness of $\reps(x,\cU)$ at one point $x_0$ and one $\varepsilon$ implies 
finiteness 
of $r_{\varepsilon'}(x,\cU)$ for all $x \in \X$ and for all $\varepsilon' \geq \varepsilon$.
In that case we obtain Lip\-schitz-continuity of the radius, as shown in the following result.

\begin{lemma}
Let an uncertain optimization problem $\rm (P(\xi),\xi\in\cU)$ be given.
Suppose there exists $x_0 \in \R^n$
such that $\reps(x_0,\cU) < +\infty.$ Then, $r_{\varepsilon}(x,\cU) < +\infty$ for all
$x \in \R^n$ and for all $\varepsilon' \geq \varepsilon$.
In such a case, the function $\R^n \ni x \, \longmapsto r_{\varepsilon'}(x,\cU)$ 
is Lipschitz-continuous with Lipschitz constant $L=1$ for every $\varepsilon' \geq \varepsilon$.
\end{lemma}

\begin{proof}
Take $x \in \R^n$ and $\xi \in \cU.$ Let $y \in \Geps(\xi)$ such that $d(x_0,\Geps(\xi)) = d(x_0,y).$ We have that
\[d(x,\Geps(\xi))  \leq  d(x,y) \leq d(x,x_0) + d(x_0,y) = d(x,x_0) + d(x_0,\Geps(\xi)) \]
Hence,
\[
\max_{\xi \in \cU} d(x,\Geps(\xi)) \leq d(x,x_0) + \max_{\xi \in \cU} d(x_0,\Geps(\xi)) < +\infty,
\]
and therefore, $\reps(x,\cU)$ is finite everywhere. Since $r_{\varepsilon'}(x,\cU) \leq \reps(c,\cU)$ for
all $\varepsilon' \geq \varepsilon$ we also have finiteness if we increase $\varepsilon$.
\medskip

We now show that $\reps(\cdot,\cU)$ is also Lipschitz-continuous. Let $\delta > 0$, and let  $x,x' \in \R^n.$
Take $\xi^*$ such that
\[
\delta + d(x,\Geps(\xi^*)) \geq \reps(x,\cU).
\]
Since $\Geps(\xi^*)$ is closed,
take also $y' \in \Geps(\xi^*)$ such that $d(x',\Geps(\xi^*)) = d(x',y').$

Then,
\begin{eqnarray*}
\reps(x,\cU) - \reps(x',\cU) & \leq & \delta +  d(x,\Geps(\xi^*))-d(x',\Geps(\xi^*)) \\
 & \leq & \delta +  d(x,y') - d(x',y') \\
 & \leq & \delta + d(x,x').
\end{eqnarray*}
Since this inequality holds for any $\delta > 0,$ we obtain
$\reps(x,\cU) - \reps(x',\cU) \leq d(x,x')$, 
hence the function $r(\cdot,\cU)$ is Lipschitz-continuous with Lipschitz constant 1.
\end{proof}

In what follows we assume finiteness of the optimal value of {\Peps}, and thus Lipschitz-continuity of $\reps(\cdot,\cU).$
Hence, {\Peps} may be solved by using standard Lipschitz optimization methods \cite{Yaro}.
\medskip

For a given $x \in \R^n$ let us call $\xi \in \cU$ a
\textit{worst-case scenario} with
respect to $x$ (and $\cU$) if
\[ d(x, \Geps(\xi))=\reps(x, \cU) \]
and let $WC_\varepsilon(x,\cU)$ be the set of all
worst-case scenarios, i.e., scenarios  $\xi \in \cU$ yielding the maximal recovery distance
for the solution $x$. Under certain assumptions, 
any optimal solution $x^*$ to {\Peps} has a set
$WC_\varepsilon(x^*,\cU)$ with at least two elements, as shown in the following result.
\medskip

\begin{lemma}
\label{le:geq2}
Let an uncertain optimization problem $\rm (P(\xi), \xi\in\cU)$ be given.
Suppose that $\cU$ is finite (with at least two elements) and $\X=\R^n$.
Fix some $\varepsilon$ and assume that
{\Peps} attains its optimum at some $x^* \in \R^n.$
Then, $|WC_\varepsilon(x^*, \cU)| \geq 2$.
\end{lemma}

\begin{proof}
Finiteness of $\cU$ implies that the maximum of $d(x^*,\Geps(\xi))$ must be attained at some $\xi.$
Hence, $|WC_\varepsilon(x^*, \cU)| \geq 1$.

In the case that $\reps(x^*,\cU)=0$, we have that $WC_\varepsilon(x^*,\cU)=\cU$. Thus, let $\reps(x^*,\cU)>0$.

In the case that $WC_\varepsilon(x^*, \cU)=\{ \xi^*\}$ for only one scenario $\xi^* \in \cU,$
we 
can construct a contradiction by finding a different $x$ with a better radius:
Take $y^* \in \Geps(\xi^*)$ such that $d(x^*,y^*) = d(x^*,\Geps(\xi^*)),$ and, for $\lambda \in [0,1],$ define $x_\lambda$ as
\[
x_\lambda = (1-\lambda)x^* + \lambda y^*.
\]
Since, by assumption, $WC_\varepsilon(x^*, \cU)=\{ \xi^*\}$  and $\cU$ is finite, there exists $\delta > 0$ such that
\[
d(x^*,\Geps(\xi)) < d(x^*,\Geps(\xi^*)) - \delta \qquad \forall \xi \in \cU, \, \xi \neq \xi^*.
\]
Let us show that, for $\lambda$ close to zero, $x_\lambda$ has a strictly better objective value than $x^*,$ which would be a contradiction.
First we have
\begin{eqnarray*}
d(x_\lambda,\Geps(\xi^*)) & \leq & d(x_\lambda,y^*) \\
 & = & (1-\lambda) \|x^*-y^*\| = (1-\lambda) d(x^*,\Geps(\xi^*))\\
& < & d(x^*,\Geps(\xi^*)) \text{ for } \lambda > 0.
\end{eqnarray*}
For the remaining scenarios $\xi \neq \xi^*,$
\begin{eqnarray*}
d(x_\lambda,\Geps(\xi)) & \leq &  \inf_{y \in \Geps(\xi)} \big\{ \|x_\lambda -x^*\| + \|x^* - y\| \big\}\\
& = & \inf_{y \in \Geps(\xi)} \big\{ \lambda\|x^* - y^* \| + \|x^* - y\| \big\}\\
& = & \lambda \|x^*-y^*\| + d(x^*,\Geps(\xi)) \\
& < & \lambda \|x^*-y^*\| + d(x^*,\Geps(\xi^*)) - \delta\\
& < & d(x^*,\Geps(\xi^*)) \text{ for } \lambda<\frac{\delta}{\|x^* - y^*\|}.
\end{eqnarray*}
Hence, for $0<\lambda<\frac{\delta}{\|x^* - y^*\|}$,
we would have that
\[
\max_{\xi \in \cU} d(x_\lambda,\Geps(\xi)) < d(x^*,\Geps(\xi^*)) = \max_{\xi \in \cU} d(x^*,\Geps(\xi)),
\]
contradicting the optimality of $x^*.$
\end{proof}

\medskip

If the finiteness assumption of Lemma \ref{le:geq2} is dropped, not much can be said about the cardinality 
of $WC_\varepsilon(x,\cU),$ since this set can be empty or a singleton:
\medskip

\begin{example}
Let $\cU = \{-1,1\}\times [1,\infty),$ and let $F(x,(\xi_1,\xi_2)) = (x-\xi_1)(\xi_2x-\xi_1\xi_2 + \xi_1)$.
Let $f(x)=const$ and choose $\varepsilon > const$.
It is easily seen that 
\begin{equation}
\label{eq:compacts}
\begin{array}{rcl}
\G_{\varepsilon}(-1,\xi_2) = \F(-1,\xi_2) &= &[-1,-1+\frac{1}{\xi_2}] \\
\G_{\varepsilon}(1,\xi_2) = \F(1,\xi_2) &=& [1-\frac{1}{\xi_2},1]
\end{array}
\end{equation}
For $x = 0,$ $\reps(x,\cU) = 1,$ but there is no $\xi\in\cU$ with $d(x,\Geps(\xi)) = 1$.
In other words, $WC_\varepsilon(0,\cU) = \emptyset.$

\end{example}

\subsection{Solving {\Peps} for a finite uncertainty set $\cU$}

\label{sec-finite}
In this section we assume that $\cU$ is finite, $\cU= \{\xi^1,\ldots,\xi^N\}.$ This simplifies the 
analysis, since
we can explicitly search for a solution $y^k=y(\xi^k)$ for every scenario $\xi^k \in \cU$.
Using the $y^k$ as variables we may formulate {\Peps} as
\begin{equation}\label{form:recfeas-finite}
\begin{array}{llll}
 \min &  r \\
\mbox{s.t. } & F(y^k,\xi^k) & \leq  & 0 \mbox{ for all } k=1,\ldots,N\\
& f(y^k,\xi^k) & \leq  & \varepsilon \mbox{ for all } k=1,\ldots,N\\
&d(x,y^k) & \leq & r \mbox{ for all } k=1,\ldots,N\\
& x \in \X, r \in \R \\
& y^k \in \X && \phantom{r} \mbox{ for all } k=1,\ldots,N.
\end{array}
\end{equation}

We can write {\Peps} equivalently as
\begin{equation*}
\min_{x \in \X} \max_{1 \leq k \leq N} d(x,\Geps(\xi^k)).
\end{equation*}
Assuming that the distance used is the Euclidean $d_2(\cdot,\cdot)$, 
the function $x\, \longmapsto \max_k d_2(x,\Geps(\xi^k))$ is known to be \textit{d.c.} for closed sets $\Geps$ \cite{dcprog}, i.e., it can be
written as a difference of two convex functions,
and then the powerful tools of d.c. programming may be used to find a globally optimal solution if
{\Peps} is low-dimensional \cite{BLanqueroCH}, or to design heuristics for more general cases
\cite{LeThi}. 

\subsubsection{Convex programming problems}
\label{sec-convex-finite}
We start with optimization problems $\rm P(\xi)$ that have convex 
sets $\Geps(\xi)$ for all $\xi \in \cU$. This is the case if the functions $F$ and $f$ 
of $\rm P(\xi)$ are quasiconvex for all fixed scenarios $\xi$, and $\X$ is convex.
We furthermore assume that $d$ is convex, which is the case, e.g., when $d$ has been derived from a norm, i.e. $d(x,y)=\|y-x\|$ 
for some norm $\| \cdot \|$.

Let us fix $\xi$. Then the function $\R^n \ni x \longmapsto d(x,\F(\xi))$
describes the distance between a point and a convex set 
and is hence convex. 
We conclude that also $\reps(x,\cU)$ is convex, being the maximum of a finite set of convex functions.
\medskip

\begin{lemma}
\label{lem-convex}
Consider an uncertain optimization problem $\rm (P(\xi), \xi \in \cU)$ with quasiconvex objective
function $f(\cdot,\xi)$ and quasiconvex constraints $F(\cdot,\xi)$ for any fixed $\xi$. Let $\X \subseteq \R^n$ be 
convex, $\cU$ be a finite set and $d$ be convex.
Then problem {\Peps} is a convex optimization problem.
\end{lemma}
\medskip

In order to solve {\Peps} one can hence apply algorithms suitable for convex
programming, e.g., subgradient or bundle methods \cite{Bundle-Book,HUL}.
In particular, if {\Peps} is unconstrained in $x$, a necessary 
and sufficient condition for a point $x^*$  to be an optimal solution is
\[ 0 \in \partial(\reps(x^*, \cU)), \]
i.e., if $0$ is contained in the subdifferential of $\reps(\cdot,\cU)$ at the point $x^*$.
By construction of $\reps(\cdot,\cU),$
we obtain
\[ 0 \in {\rm conv} \left\{ \partial d(x^*, \Geps(\xi)): \xi \in WC_\varepsilon(x^*,\cU) \right\} \]
where $WC_\varepsilon(x^*,\cU)$ is the set of
worst-case scenarios (see \cite{HUL}), and $\partial d(x^*, \Geps(\xi))$ is the subdifferential of $d(\cdot,\Geps(\xi))$ at $x^*$.

Now, $\partial d(x^*, \Geps(\xi))$ can be written in terms  of the subdifferential of the distance used,
see \cite{CaFl02}, where also easy representations for polyhedral norms or the Euclidean norm
are presented.  Although we do not know much a priori about the number of worst-case
scenarios, we do not need to investigate all possible subsets but
may restrict our search to sets which do not have
more than \mbox{$n+1$} elements as is shown in our next result. 
This may be helpful in problems with a large number of scenarios but low dimension $n$ for the decisions.
\medskip

\begin{theorem}
\label{theo-helly}
Let $\cU$ be finite with cardinality of at least $n+1$.
Let $\X = \R^n$. Suppose {\Peps} attains its optimum at some $x^*$, and that
for each $\xi$ the functions
$F(\cdot,\xi)$ and $f(\cdot,\xi)$ are quasiconvex. Let $d$ be convex.
Then there exists a subset $\overline{\cU} \subseteq \cU$
of scenarios with $ 2 \leq |\overline{\cU}| \leq n+1$ such that
\[\reps^*(\cU) = \reps(x^*, \cU)=\reps(x^*, \overline{\cU})=\reps^*(\overline{\cU}). \]
\end{theorem}

\begin{proof}
Let $x^*$ be optimal for {\Peps}. The result is trivial if $\reps(x^*,\cU) = 0$: 
take any collection of $n+1$  scenarios.
Hence, we may assume  $\reps(x^*,\cU) > 0,$ which implies that $x^*$ does not belong to all sets $\Geps(\xi).$

By Lemma \ref{le:geq2},  $|WC_\varepsilon(x^*,\cU) | \geq 2$.
If $|WC_\varepsilon(x^*,\cU) | \leq n+1,$ then we are done. Otherwise,
$|WC_\varepsilon(x^*,\cU) | > n+1,$  we have by the optimality of $x^*$ and convexity of the functions $d(\cdot,\Geps(\xi)),$ that
\[ 0 \in {\rm conv} \left\{ \partial d(x^*, \Geps(\xi)): \xi \in WC_\varepsilon(x^*,\cU) \right\} \]
By Carath\'eodory's theorem, $WC_\varepsilon(x^*,\cU)$ contains a subset $\overline{\cU},$
$1 \leq |\overline{\cU}| \leq n+1$
such that
$0 \in {\rm conv} \left\{ \partial d(x^*, \Geps(\xi)): \xi \in \overline{\cU}\right\}.$
Such $\overline{\cU}$ clearly satisfies the conditions stated.
\end{proof}

\subsubsection{Problems with linear constraints and polyhedral norms as recovery costs}
\label{sec-poly-finite}

As in the section before, we assume a finite uncertainty set
$\cU=\{\xi^1,\ldots,\xi^N\}$.
Let us now consider the case that
all sets $\Geps(\xi^k)$, $k=1,\ldots,N$ are polyhedral sets.
More precisely, we consider problems of type
\begin{align*}
\mbox{P($\xi$)} \hspace{1cm} \min\ &f(x,\xi):= c(\xi)^t x \\
\mbox{s.t. }   & F(x,\xi):=A(\xi) x - b(\xi)  \leq 0 \\
& x \in \X
\end{align*}
with a finite uncertainty set $\cU=\{\xi^1,\ldots,\xi^N\}$, linear
constraints $A(\xi)x \le b(\xi)$ for every $\xi \in \cU$, a linear objective
function $c(\xi)^tx$ and a polyhedron $\X.$
\medskip

Furthermore, let us  assume that the distance $d$ is induced by a block norm
$\| \cdot \|$, i.e., a norm whose unit ball is a polytope, see~\cite{ww85,witzgall}.
The most prominent examples for block norms are the Manhattan ($d_1$) and the maximum ($d_\infty$) norm, which
both may be suitable to represent recovery
costs: In the case that the recovery costs are obtained by adding single costs of each
component, the Manhattan norm
is the right choice. The maximum norm may represent the recovery
time in the case that a facility has to be moved along each coordinate (or a schedule
has to be updated by a separate worker in every component) and
the longest time determines the time for the complete update.

We also remark that it is possible to approximate any given norm arbitrarily close
by block norms, since the class of block norms is a dense subset
of all norms, see \cite{ww85}.
Thus, the restriction to the class of block norms may not be a real
restriction in a practical setting.
\medskip

The goal of this section is to show that under the assumptions above,
{\Peps} is a linear program.
\medskip

We start with some notation. Given a norm $\| \cdot \|$, let
\[ B =\{x \in \R^n: \| x\| \leq 1 \} \]
denote its unit ball. Recall that the unit ball of a block norm $\| \cdot \|$ is a
full-dimensional convex polytope which is symmetric with respect to the origin. Since such
a polytope has a finite number $S$ of extreme points, we may denote in
the following the extreme points of  $B$  as
\[
 \mathrm{Ext}(B) = \{e_i : 1\leq i \leq S\}.
\]
Since $B$ is symmetric  with respect to the origin,
 $S\in\N$ is always an even number and for any $e_i\in \mathrm{Ext}(B)$ there exists another
$e_j\in\mathrm{Ext}(B)$ such that $e_i=-e_j$. Its dual (or polar) norm defined
as $\|x\|_0:=\max\{x^ty: \|y \| \leq 1\}$ has the unit ball
\[
B^0=\{x \in \R^n: x^ty \leq 1 \mbox{ for all } y \in B\}.
\]
It is known that $B^0$ is again a polyhedral norm with extreme points
\[
\mathrm{Ext}(B^0) = \{e^0_i : 1\leq i \leq S^0\},
\]
where $S^0$ is the number of facets of $B$ (see, e.g., \cite{Roc70}).
\medskip

The following property is crucial for the linear
programming formulation of {\Peps}. It shows that it is sufficient
to consider only the extreme points $\mathrm{Ext}(B)$ of either the unit ball $B$ of the block norm,
or of the unit ball $B^0$ of its polar norm
in order to compute $\Vert x\Vert$ for any point $x \in\R^n$.
\medskip

\begin{lemma}[\cite{ww85}]\label{thm:poly-distance}
Let $\mathrm{Ext}(B) = \{e_i : 1\leq i \leq S\}$ be the extreme points of a block norm $\| \cdot \|$ with unit ball
$B$ and let $\mathrm{Ext}(B^0) = \{e^0_i : 1\leq i \leq S^0\}$ be the extreme points of its polar norm with unit ball
$B^0$. Then $\| \cdot \|$ has the following two characterizations:
\begin{align*}
  \| x \|= \min\left\lbrace \sum_{i=1}^S  \beta_i : x = \sum_{i=1}^S \beta_i e_i,~\beta_i\geq 0 \:\forall\: i=1,\ldots,S\right\rbrace
\end{align*}
and
\begin{align*}
  \| x \|= \max_{i=1,\ldots,S^0} x^t e_i^0.
\end{align*}
\end{lemma}
\medskip

Lemma~\ref{thm:poly-distance} implies that we can compute $\Vert x-y\Vert$ for any pair $x,y\in\R^n$ by linear 
programming.
Thus, our assumptions on the sets $\Geps(\xi^k)$ and Lemma~\ref{thm:poly-distance} give rise to the following linear 
formulations of {\Peps}, if $\X$ is a polyhedron:
\begin{alignat}{6} 
 \min \quad & r \nonumber \\
\mbox{s.t. } & A(\xi^k)y^k & \quad \leq \quad & b(\xi^k) &\mbox{ for all } k=1,\ldots,N \label{lp1}\\
& c(\xi^k)^t y^k & \quad \leq \quad & \varepsilon &\mbox{ for all } k=1,\ldots,N
\label{lp1b}\\
 & y^k - x  & \quad = \quad &   \sum_{i=1}^S \beta^k_i e_i  &\mbox{ for all } k=1,\ldots,N \label{lp2}\\
   & \sum_{i=1}^S  \beta^k_i & \quad \leq \quad& r  &\mbox{ for all } k=1,\ldots,N \label{lp3}\\
& x, y^k & \quad \in \quad & \X &\mbox{ for all } k=1,\ldots,N \label{lp6} \\
& r,\beta^k_i & \quad \geq \quad & 0 &   \mbox{ for all } k=1,\ldots,N,\ i=1,\ldots,S \label{lp5}
\end{alignat}
Note that constraints (\ref{lp1})  and (\ref{lp1b}) are just the definition of the sets
$\Geps(\xi^k)$. Furthermore, (\ref{lp2}) and (\ref{lp3}) 
together ensure 
that $\Vert x - y^k\Vert \leq r$ for all
$k=1,\ldots,N$. Hence, the linear program is
equivalent to the formulation~\eqref{form:recfeas-finite}
for a finite set of scenarios each of them having
a polyhedron as feasible set and if a block norm is used as distance measure.
In this case we have hence shown that {\Peps} can be formulated as
a linear program. In order to use the second characterization of block
norms in Lemma~\ref{thm:poly-distance} we replace (\ref{lp2}) and (\ref{lp3}) by
\begin{equation}
\left(e_i^0\right)^t (y^k-x) \leq r \quad \mbox{ for all }  k=1,\ldots,N,\ i=1,\ldots,S^0 
\end{equation}
to ensure that the value of $\|x-y^k\|$ is correctly computed.
We summarize our findings in the following result.
\medskip

\begin{theorem}
\label{theo-finite}
Consider an uncertain linear optimization problem $\rm (P(\xi), \xi \in \cU)$ 
Let $\cU=\{\xi^k: k=1,\ldots,N\}$
be a finite set and let $d$ be induced by a block norm. 
Let $\X \subseteq \R^n$ be a polyhedron.
Then {\Peps} can be solved by linear programming.

If the number of constraints defining $\X$, and either the number  
of extreme points of $B$ or the number of facets of $B$ depend at most polynomially on the dimension $n$, then 
{\Peps} can be solved in polynomial time.
\end{theorem}

We note that block norms may be generalized to the broader class of
polyhedral gauges where the symmetry assumption on the unit ball is
dropped (see e.g., \cite{nickel09}). Nevertheless it is readily shown that
Lemma~\ref{thm:poly-distance} applies to polyhedral gauges as well. Hence, it
follows that Theorem~\ref{theo-finite} also holds for distance
functions derived from polyhedral gauges. 

\subsubsection{Problems with hyperplanes as feasible sets}

We consider a special case in which {\Peps} can be rewritten as
a linear program, even though the distance measure does not need to be derived
from a block norm, namely if the sets $\Geps(\xi)$ are all hyperplanes or
halfspaces. Before we show the resulting linear program for this case, 
we consider some situations in which this happens:

\begin{example}
Let $d$ be a distance derived from a norm, and let $\X=\R^n$.
\begin{enumerate}
\item For feasibility problems of type
\begin{align*}
\mbox{P$(a,b)$} \hspace{1cm} \min\ &const \\
\mbox{s.t. }  & F(x,(a,b)):=a^t x - b = 0 
\end{align*}
with $\cU = \{(a^1,b^1),\ldots,(a^N,b^N)\}$, 
$a^1,\ldots,a^N \neq 0$ we obtain $\Geps(a^k,b^k)=\{x: {a^k}^t x-b^k=0\}$
for all $\varepsilon > const$.
\item The same holds for problems
\[ P(\xi) \ \ \min\{f(x,\xi): x \in \F(\xi)\} \mbox{ for } \xi \in \cU\]
if $\F(\xi)$ is a hyperplane for each $\xi \in \cU$ and
$\varepsilon > f(x,\xi)$ for all $x \in \F(\xi)$.
\item For unconstrained uncertain linear optimization of the form
\[ P(\xi)  \ \ \min\{c(\xi)^tx: x \in  \R^n\} \]
the resulting sets $\Geps(\xi^k)=\{x: c(\xi)^tx \leq \varepsilon\}$ 
are halfspaces.
\end{enumerate}
\end{example}

Let us first consider the case of hyperplanes:
For $\xi^k=(a^k,b^k)$, let $\Geps(\xi)=H_{a^k,b^k} = \{x\in\R^n : {a^k}^tx = b^k\}$ be a hyperplane.
Then {\Peps} is given by
\begin{align*}
 \min \ &r\\
\mbox{s.t. } &d(x,H_{a^k,b^k} )\leq r \mbox{ for all } k=1,\ldots,N\\
 &x \in \R^n, r \in \R,
\end{align*}
Recall the point-to-hyperplane distance \cite{PlCa01}
\[d(x,H_{a,b}) = \frac{|a^tx - b|}{\Vert a \Vert^\circ},\]
where $\|\cdot\|^\circ$ denotes the dual norm to $\|\cdot\|$.
As the values of $\Vert a^k \Vert^\circ$
can be precomputed and the absolute value linearized, we gain a linear program
\begin{align}
 \min \ &r\nonumber\\
\mbox{s.t. } & -r \le \frac{{a^k}^tx - b}{\Vert a^k \Vert^\circ} \leq r \mbox{ for all } k=1,\ldots,N\label{recfeas-linprog-2}\\
 &x \in \R^n, r \in \R.\nonumber
\end{align}
\medskip

For halfspaces 
$\Geps(\xi^k)=H^+_{a^k,b^k} = \{x\in\R^n : {a^k}^tx \leq b^k\}$ 
instead of hyperplanes, the distance is given by
\[d(x,H^+_{a,b}) = \frac{|a^tx - b|^+}{\Vert a \Vert^\circ},\]
where $|a^tx - b|^+ = \max\{a^tx - b, 0\}$, resulting in the linear program
\begin{align}
 \min \ &r\nonumber\\
\mbox{s.t. } & \frac{{a^k}^tx - b}{\Vert a^k \Vert^\circ} \leq r \mbox{ for all } k=1,\ldots,N
\label{recfeas-linprog-3}\\
& r\ge 0 \nonumber\\
 &x \in \R^n, r \in \R \nonumber.
\end{align}

\begin{theorem}
Consider an uncertain optimization problem 
with finite uncertainty set and sets $\Geps(\xi)$ that are hyperplanes or halfspaces. Let $\X=\R^n$ and let $d$ be derived from a norm $\|\cdot\|$.
Then {\Peps} can be formulated as linear program (see \eqref{recfeas-linprog-2} and \eqref{recfeas-linprog-3}) and be solved in
polynomial time, provided that the dual norm of $\|\cdot\|$ can be evaluated in polynomial time.
\end{theorem}

\section{Reduction approaches}
\label{sec-infinite}

In this section we analyze recoverable-robust solutions for different uncertainty sets $\cU$, and hence write
Rec($\cU$), ${\bf f_{\cU}}$ and ${\bf r_{\cU}}$  to emphasize the uncertainty set that is considered:
\begin{eqnarray*}
\mbox{Rec($\cU$)} \hspace{5mm} & {\rm minimize} \ 
\left( {\bf f_{\cU}}(y), {\bf r_{\cU}}(x,y) \right)=
\left( \sup_{\xi \in \cU} f(y(\xi),\xi), \sup_{\xi \in \cU} d(x,y(\xi)) \right)\\
& \mbox{s.t. } F(y(\xi),\xi) \leq 0\ \ \mbox{ for all } \xi \in \cU\\
& x \in \X, y:\cU \to \X
\end{eqnarray*}
Recall that a solution $x$ is
\emph{recoverable-robust with respect to $\cU$} if there exists $y:\cU \to \X$ such that 
$(x,y)$ is Pareto-efficient for Rec($\cU$).

The main goal of this section is to reduce the set $\cU$ to a smaller (maybe even finite)
set $\cU' \subseteq \cU$, such that the set of recovery-robust solutions does not change. 
This is the case if we can extend any feasible solution $(x,y')$ for Rec($\cU'$) to a feasible
solution $(x,y)$ for Rec($\cU$) without changing the objective function values.

\begin{lemma}
\label{lem-reductionlemma}
Let $\cU' \subseteq \cU$. If for all feasible solutions $(x,y')$ of Rec($\cU'$) there exists 
$y: \cU \to \X$ such that 
\begin{itemize}
\item $(x,y)$ is feasible for Rec($\cU$), i.e., $F(y(\xi),\xi) \leq 0$ for all $\xi \in \cU$, and
\item  ${\bf f_{\cU}}(y)={\bf f_{\cU'}}(y')$ and ${\bf r_{\cU}}(x,y)={\bf r_{\cU'}}(x,y')$
\end{itemize}
then Rec($\cU$) and Rec($\cU'$) have the same recoverable-robust solutions.
\end{lemma}

\begin{proof}
Let $(x,y)$ be feasible for Rec($\cU$). Define
\[ y_{|\cU'}:\cU' \to \X \ \mbox{ through } y_{|\cU'}(\xi):=y(\xi) \ \mbox{ for all } \xi \in \cU' \] 
Then $(x,y')$ is feasible for Rec($\cU'$) and ${\bf f_{\cU'}}(y')\leq {\bf f_{\cU}}(y)$, ${\bf r_{\cU'}}(x,y') \leq {\bf r_{\cU}}(x,y)$.
Together with the assumption of this lemma Pareto optimality follows since a solution can be improved by switching between Rec($\cU$) and 
Rec($\cU'$):
\begin{itemize}
\item Let $x$ be recoverable-robust w.r.t $\cU$. Then there exists $y:\cU \to \X$ such that $(x,y)$ is Pareto efficient for Rec($\cU$). Define
$y':=y_{|\cU'}$. Then $(x,y')$ is Pareto-efficient for Rec($\cU'$): Assume that $(\tilde{x},\tilde{y}')$ dominates $(x,y')$. Due to the assumption
of this lemma there exists $(\tilde{x},\tilde{y})$ which is feasible for Rec($\cU$) and 
${\bf f_{\cU}}(\tilde{y})={\bf f_{\cU'}}(\tilde{y}')$ and ${\bf r_{\cU}}(\tilde{x},\tilde{y})={\bf r_{\cU'}}(\tilde{x},\tilde{y}')$, i.e., 
$(\tilde{x},\tilde{y})$ then dominates $(x,y)$, a contradiction.

\item Let $x$ be recoverable-robust w.r.t $\cU'$. Then there exists $y':\cU' \to \X$ such that $(x,y')$ is Pareto-efficient for Rec($\cU'$). Choose
$y$ according to the assumption of this lemma. Then $(x,y)$ is Pareto-efficient for Rec($\cU$): 
Assume that $(\tilde{x},\tilde{y})$ dominates $(x,y)$. Then $(\tilde{x},\tilde{y}_{|\cU'})$ is feasible for Rec($\cU'$) and 
${\bf f_{\cU'}}(\tilde{y}_{|\cU'})\leq{\bf f_{\cU}}(\tilde{y})$ and ${\bf r_{\cU'}}(\tilde{x},\tilde{y}_{|\cU'})\leq{\bf r_{\cU}}(\tilde{x},\tilde{y})$, i.e., 
$(\tilde{x},\tilde{y}_{|\cU'})$ then dominates $(x,y')$, a contradiction.
\end{itemize}
\end{proof}

We now use Lemma~\ref{lem-reductionlemma} to reduce the 
set of scenarios $\cU$. Our first result is similar to 
the reduction rules for set covering problems \cite{TSRB71}.

\begin{lemma}
If P($\xi^2$) is a relaxation of P($\xi^1$) for two scenarios $\xi^1,\xi^2 \in \cU$,
then Rec($\cU$) and Rec($\cU \setminus \{ \xi^2 \}$) have the same recoverable robust solutions, 
i.e., scenario $\xi^2$ may be ignored. 
\end{lemma}

\begin{proof}
We check the condition of Lemma~\ref{lem-reductionlemma}: Let $(x,y')$ be feasible for Rec($\cU \setminus \{\xi^2\}$).
Define 
\[ y:\cU \to \X \ \mbox{ through } y(\xi):=\left\{ \begin{array}{ll} y'(\xi) \mbox{ if } \xi \in \cU \setminus \{\xi^2\} \\ 
y'(\xi^1) \mbox{ if } \xi=\xi^2 \end{array} \right. \]
Then $(x,y)$ is feasible since $F(y(\xi),\xi) \leq 0$ for all $\xi \in \cU \setminus \{\xi^2\}$ and
$F(y(\xi^2),\xi^2)=F(y(\xi^1),\xi^2) \leq 0$ since $F(y(\xi^1),\xi^1)\leq 0$ and P($\xi^2$) is a relaxation of P($\xi^1$).
Furthermore, $f(y(\xi^2),\xi^2) = f(y(\xi^1),\xi^2) \leq f(y(\xi^1),\xi^1)$ implies
\[ {\bf f_{\cU}}(y)=\sup_{\xi \in \cU} f(y(\xi),\xi)= \sup_{\xi \in \cU \setminus \{\xi^2\}} f(y(\xi),\xi)= \sup_{\xi \in \cU \setminus \{\xi^2\}} f(y'(\xi),\xi)= 
{\bf f_{\cU \setminus \{\xi^2\}}}(y'). \]
Finally, $y(\xi^1)=y(\xi^2)$, hence
\[ {\bf r_{\cU}}(x,y)=\sup_{\xi \in \cU} d(x,y(\xi))= \sup_{\xi \in \cU \setminus \{\xi^2\}} d(x,y(\xi))= 
\sup_{\xi \in \cU \setminus \{\xi^2\}} d(x,y'(\xi))= \br_{\cU \setminus \{\xi^2\}} (x,y'). \]
\end{proof}

Note that depending on the definition of the optimization problem and the uncertainty set $\cU$, often large classes
of scenarios may be dropped. This is in particular the case if the sets $\F(\xi)$ are nested.
\medskip

In the following we are interested in identifying a kind of \textit{core set}
$\cU' \subseteq \cU$ containing a \textit{finite} number of scenarios which
are sufficient to consider in order to solve the recoverable-robust counterpart.
More precisely, we look for a finite set $\cU'$ such that Rec($\cU)$ 
and Rec($\cU')$ have the same recoverable-robust solutions.
\medskip

In the following we consider a polytope $\cU$ with a finite number of extreme
points $\xi^1,\ldots,\xi^N$, i.e., let
\[ \cU=conv(\cU') \ \mbox{ where } \ \cU'=\{\xi^1,\ldots,\xi^N\}. \]

Then many robustness concepts have (under mild conditions) the
following property: Instead of investigating {\it all} $\xi \in \cU$,
it is enough to investigate the extreme points $\xi^1,\ldots,\xi^N$ of
$\cU$.  
For example, for the
strictly robust counterpart RC$(\cU)$ of an uncertain optimization
problem ($P(\xi), \xi \in \cU=conv\{\xi^1,\ldots,\xi^N\}$), 
RC$(\cU)$ is equivalent to RC$(\{\xi^1,\ldots,\xi^N\})$,
if $F(x,\cdot)$ is convex for all $x \in \X$.
\medskip

Unfortunately, a similar result for
the recoverable-robust counterpart does not hold. This means that
the set of Pareto efficient solutions
of Rec($\cU'$) does in general not 
coincide with the set of Pareto efficient solutions
of Rec($\cU$)
with respect to the larger set $\cU=conv(\cU')$ as demonstrated in the
following example.

\medskip

\begin{example}\label{exa1}

Consider the following uncertain optimization problem:
\begin{align*}
\mbox{P($a_1,a_2,b$)} \hspace{1cm} \min\ & f(x_1,x_2) = const \\
\mbox{s.t. } &a_1 x_1 + a_2 x_2 - b = 0\\
& x_1,x_2\in\R,
\end{align*}
where
\[\cU = conv(\cU') \text{ with } \cU' = \{(1,0,0), (0,1,0), (1,1,2)\}.\]
Let the recovery distance
be the Euclidean distance. 
Then $x^* = (2-\sqrt2, 2-\sqrt2)$, the midpoint of the incircle of the triangle that is given by the intersections of
the respective feasible solutions, is a Pareto efficient solution of Rec($\cU')$, 
as it is the unique minimizer of the recovery distance
(see Figure \ref{example-a}).

\begin{figure}[htbp]
\centering
\subfigure[Optimal solution w.r.t. $\cU'$.]{
\label{example-a}
\includegraphics[width = 0.46\textwidth]{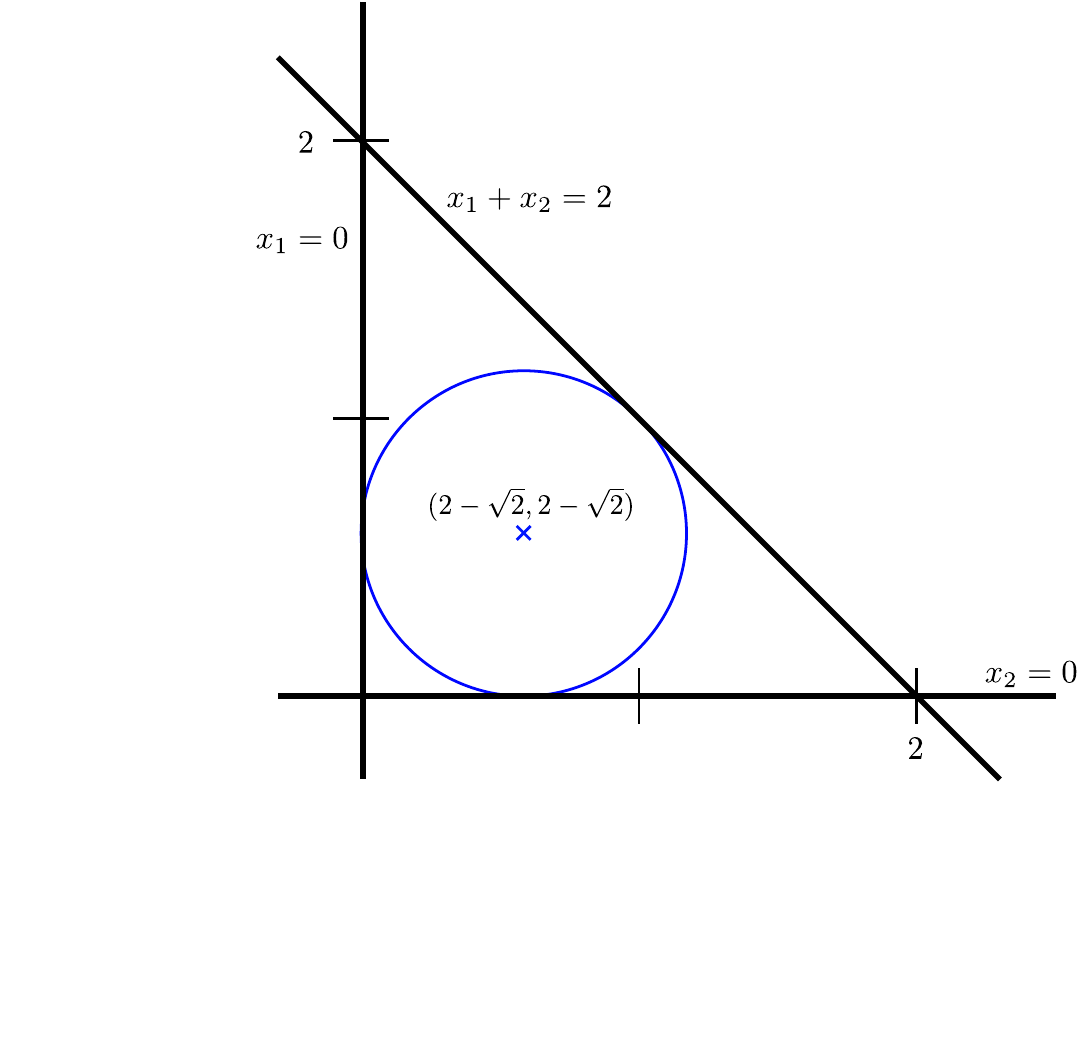} }
\subfigure[Optimal solution w.r.t. $\cU$.]{
\label{example-b}
\includegraphics[width = 0.46\textwidth]{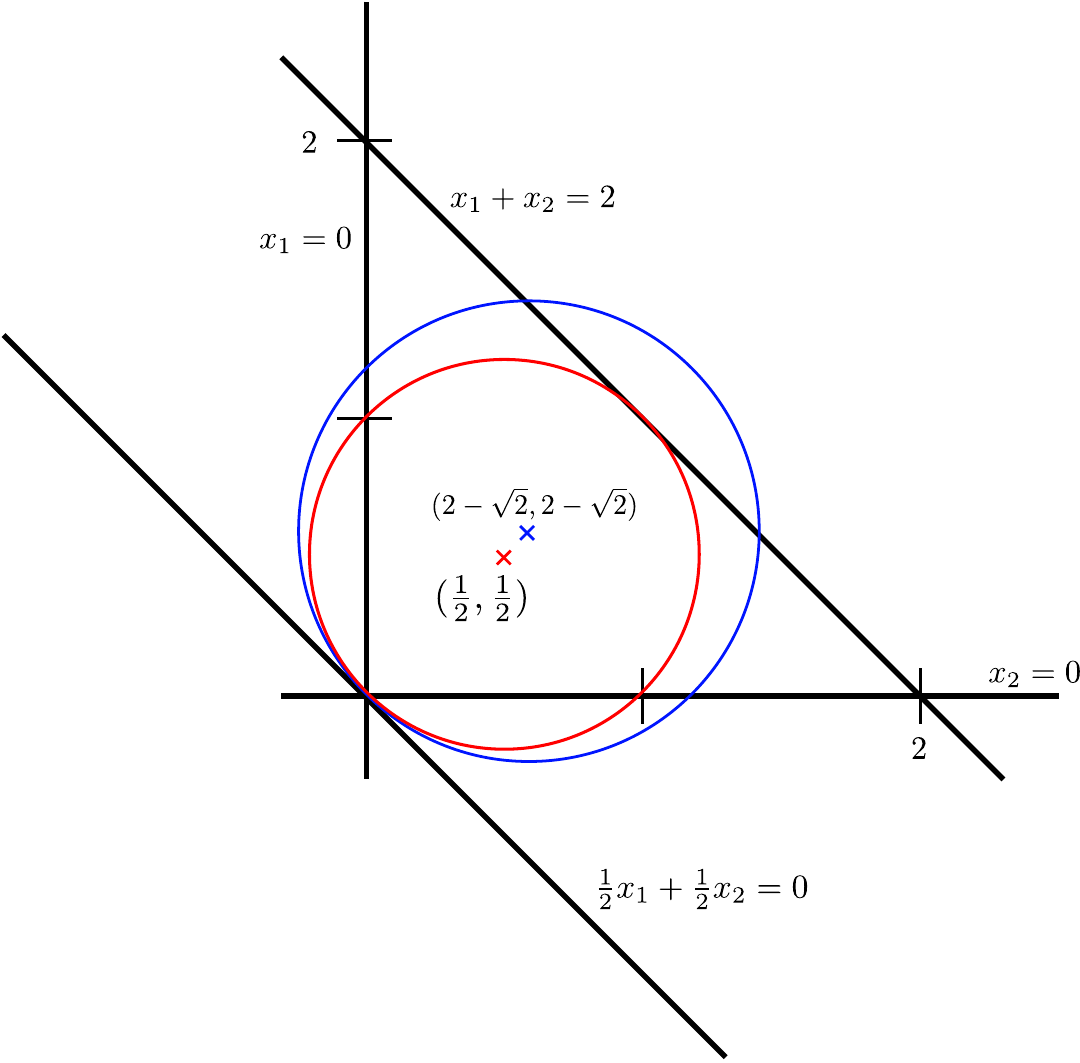} }
\caption{Rec($\cU')$ and Rec($\cU)$ may have different optimal solutions.}
\label{example-ab}
\end{figure}

On the other hand, this solution is not Pareto efficient when the convex hull of $\cU'$ is taken into
consideration. Indeed, by elementary geometry, one finds that
\begin{align*}
r(x^*,\cU) &= \sqrt 2\cdot (2-\sqrt 2) \approx 0.828,\\
r(\bar{x},\cU) &= \frac{1}{\sqrt 2} \approx 0.707,
\end{align*}
where $\bar{x} = (\frac{1}{2},\frac{1}{2})$ (see Figure~\ref{example-b}).
Therefore, solving Rec($\cU')$ does not give the set of Pareto efficient solutions for Rec($\cU)$ .

\end{example}

However, assuming more problem structure, we can give the following result.

\begin{theorem}
\label{th:quasiconvex}
Consider an uncertain optimization problem with
uncertainty set $\cU=conv(\cU')$ with $\cU':=\{\xi^1,\ldots,\xi^N\}$. 
Let $F$
consist of $m$ constraints with $F_i:\R^n \times \cU \to \R$, $i=1,\ldots,m$ and $f: \R^n \times \cU \to \R$
be jointly quasiconvex in the arguments $(y,\xi)$. Let $d(x,\cdot)$ be quasiconvex. 
Let $\X$ be convex.

Then Rec($\cU$) and Rec($\cU'$) have the same set of recoverable-robust solutions.
\end{theorem}

\begin{proof}
Let $(x,y')$ be feasible for Rec($\cU'$). We first define $y:\cU \to \X$.

Let $\xi \in \cU$. Then there exist $\lambda_i, i=1,\ldots,N$ such that 
$\xi= \sum_{i=1}^N \lambda_i \xi^i$ with $\sum_{i=1}^N \lambda_i = 1$ and $\lambda_i \geq 0$.
We set $y(\xi) := \sum_{i=1}^N \lambda_i y'(\xi^i)$. 
Note that this implies $y(\xi^i)=y'(\xi^i)$ for all $i=1,\ldots,N$.
We now check the conditions of Lemma~\ref{lem-reductionlemma}.
\medskip

For every constraint $k=1,\ldots,m$ the joint quasiconvexity implies that
\[ F_k (y(\xi), \xi) = F_k\left(\sum_{i=1}^N \lambda_i y(\xi^i), \sum_{i=1}^N \lambda_i \xi^i \right) \le \max_{i=1,\ldots,N} F_k(y(\xi^i),\xi^i) \le 0 \ \forall k=1,\ldots,m,\]
where the last inequality holds since $y(\xi^i)=y'(\xi^i)$ and $(x,y')$ is feasible for Rec($\cU'$). We hence have that $(x,y)$ is feasible for 
Rec($\cU$).

Analogously, joint quasiconvexity of $f$ implies 
$f( y(\xi), \xi) \le \max_{i=1,\ldots,N} f(y(\xi^i),\xi^i)$ for all $\xi \in \cU$, hence 
\[ {\bf f_{\cU}}(y) = \sup_{\xi \in \cU} f(y(\xi),\xi) = \max_{\xi \in \cU'} f(y(\xi),\xi) = \max_{\xi \in \cU'} f(y'(\xi),\xi) = {\bf f_{\cU'}}(y'). \]
Finally, for the recovery distance $d$ we assumed quasiconvexity in its second argument  which implies
$d(x,y(\xi)) \le \max_{i=1,\ldots,N} d(x,y(\xi^i))$, hence
\[ {\bf r_{\cU}}(x,y) = \sup_{\xi \in \cU} d(x,y(\xi)) = \max_{\xi \in \cU'} d(x,y(\xi)) = \max_{\xi \in \cU'} d(x,y'(\xi)) = {\bf r_{\cU'}}(x,y'). \]

\end{proof}

\medskip

An important particular case of Theorem \ref{th:quasiconvex} is the case in which
\[F(x,\xi) = G(x) -b(\xi)\] for a convex $G$ and concave $b$ (i.e., the uncertainty
is in the right-hand side), since $F$ is then jointly quasiconvex in $(x,\xi)$.

\medskip

\begin{corollary}
\label{theo-right-hand}
Let $\rm(P(\xi), \xi \in \cU)$ be an uncertain optimization problem with
uncertainty set $\cU=conv(\cU')$ with $\cU':=\{\xi^1,\ldots,\xi^N\}$. Let
$F(x,\xi)=G(x)-b(\xi)$ with
a convex function $G:\R^n \to \R^m$ and a concave function
$b(\xi):\R^M \to \R^m$. Let $f(x,\xi)$ be jointly quasiconvex, $\X$ be convex, and let $d(x,\cdot)$ be quasiconvex. Then 
Rec($\cU)$ and Rec($\cU')$ have the same recoverable-robust solutions.
\end{corollary}

\medskip

We remark that $G$ must not depend on the scenario $\xi$. 
Example~\ref{exa1} shows that Corollary~\ref{theo-right-hand}
is not even true for a linear function $F(x,\xi)=A(\xi)x-b(\xi)$: 
If the matrix $A$ is dependent on $\xi$, we cannot conclude that 
Rec($\cU)$  and Rec($\cU')$ have the same recoverable-robust solutions.

Note that Corollary~\ref{theo-right-hand} applies in particular for the special case
where
$b(\xi)=\xi$,
i.e., for uncertain convex optimization problems of the type
\begin{equation}
\label{conv2}
 {\rm P}(b)  \ \min_{x \in \R^n} \{f(x): G(x)\leq b\}.
\end{equation}
In particular we know for \rm P($b$) that the center with respect to some finite set $\cU'$ solves the
uncertain problem with respect to $\cU=conv(\cU')$.

This means we can use the finite set $\cU'$ instead of $\cU$ when solving (Rec) if the
conditions of the previous theorem apply. This is summarized next.

\medskip

\begin{corollary}
\label{cor-right-hand}
Let $\rm(P(\xi), \xi \in \cU)$ be an uncertain optimization problem with
uncertainty set $\cU=conv(\cU')$ with $\cU':=\{\xi^1,\ldots,\xi^N\}$ and with
constraints $F(x,\xi)=G(x)-b(\xi)$ with
a convex function $G:\R^n \to \R^m$ and a concave 
function $b(\xi):\R^M \to \R^m$. Let $\X\subseteq \R^n$ be convex, let $f$ be jointly convex, and let $d(x,\cdot)$ be convex.
Then (Rec) can be formulated as the following convex biobjective program:
\begin{equation}
\label{right-hand}
\begin{array}{lllll}
 \min & (r,z) \\
\mbox{s.t. } & G(y^k) & \leq  & b(\xi^k) & \mbox{ for all } k=1,\ldots,N\\
& d(x, y^k ) & \leq & r & \mbox{ for all } k=1,\ldots,N\\
& f(y^k, \xi^k ) & \leq & z & \mbox{ for all } k=1,\ldots,N\\
&  x, y^k & \in & \X & \mbox{ for all } k=1,\ldots,N \\
& r,z &\in &\R \\
\end{array}
\end{equation}
\end{corollary}
\medskip

Combining this corollary with Theorem~\ref{theo-finite} from
Section~\ref{sec-poly-finite}, we obtain the following result: The
recoverable-robust counterpart of an optimization problem with
convex uncertainty which is only in its right-hand side and with
polyhedral uncertainty set can be formulated as a linear program if a
block norm is used to measure the recovery costs.  In particular,
the recoverable-robust counterpart of such a linear program under
polyhedral uncertainty sets and block norms as distance functions remains a
linear program.

\medskip

\begin{theorem}\label{theo-together}
Let $\rm(P(\xi), \xi \in \cU)$ be an uncertain linear program with concave uncertainty only in the right-hand side, and $\cU=conv(\cU')$ with $\cU':=\{\xi^1,\ldots,\xi^N\}$. 
Let $d$ be derived from a block norm. Then, (Rec) can be formulated as a linear biobjective program.

If the terms defining $\X$ and either the number
of extreme points or the number of facets of the unit ball of the block norm 
depend at most polynomially on the
dimension $n$, then the problem {\Peps} be solved in polynomial time.
\end{theorem}

\begin{proof}
According to Theorem~\ref{theo-right-hand} we can replace $\cU$ by the finite set
$\cU'$ in the recoverable-robust counterpart, i.e., we consider Rec($\cU'$)
instead of Rec($\cU$).
We are hence left with a problem for which the assumptions
of Theorem~\ref{theo-finite} are satisfied yielding a formulation as linear program.
\end{proof}

\medskip

Note that many practical applications satisfy the conditions
of Theorem~\ref{theo-together}. Among these are scheduling and timetabling
problems where the uncertainty is the length of the single tasks to be completed
and hence in the common linear formulations in the right-hand side. We refer to 
\cite{Atmos2010-goerigk} for applications in
timetabling, to \cite{herroelen2005} for project scheduling, to \cite{Savel09} for container 
repositioning, and to
\cite{rrcg} for knapsack problems.

\section{Numerical experiments}
\label{sec:exp}

In the following, we analyze the difference between our scalarization 
{\Peps} and the ''classic'' scalarization {\Pdelta}
to calculate the Pareto front of an uncertain portfolio optimization problem 
using computational experiments.

\subsection{Problem setting}

We consider a portfolio problem of the form
\begin{align*}
\max\ &\sum_{i=1}^n p_i x_i \\
\text{s.t. } & \sum_{i=1}^n x_i = 1\\
& x\ge 0
\end{align*}
where variable $x_i$ denotes the amount of investment in opportunity $i\in\{1,\ldots,n\}$ with profit $p_i$. We assume that profits are uncertain and stem from a finite uncertainty set $\cU = \{p^1,\ldots,p^N\} \subseteq \R^n_+$. 
The biobjective recoverable-robust model we would like to solve is the following:
\begin{align*}
&(\max\ z,\ \min\ d) \\
\text{s.t. } & z \le \sum_{i=1}^n p^k_i x^k_i & \forall k=1,\ldots,N \\
& \sum_{i=1}^n x_i = 1\\
& \sum_{i=1}^n x^k_i = 1 & k=1,\ldots,N\\
& \sum_{i=1}^n (x_i - x^k_i)^2 \le d & \forall k=1,\ldots,N \\
& x,x^k \ge 0
\end{align*}
In this setting, we would like to fix some choice of investment $x$ now, but can modify it, once the scenario becomes known. Our aim is to maximize the resulting worst-case profit, and also to minimize the modifications to our investment, which we measure by using the Euclidean distance.

We compare the two $\varepsilon$-constraint approaches, where either a fixed budget on $d$ is given {\Pdelta}, or a budget on 
$z$ is given {\Peps}.

Moreover, we consider the following iterative \emph{projection method} as another solution approach to {\Peps}
It is based on the method of alternating projections. Say we have some candidate solution $x$ available. For every scenario $k$, we want to find a solution $x^k$ that is as close to $x$ as possible, and also respects a desired profit bound $P$. The resulting problems are independent for every $k$. For a fixed $k$, it
can be formulated as the following quadratic program:
\begin{align*}
\min\ &\sum_{i=1}^n (x^k_i - x_i)^2 \\
\text{s.t. } & \sum_{i=1}^n x^k_i = 1 \\
& \sum_{i=1}^n p^k_i x^k_i \ge P \\
& x^k \ge 0
\end{align*}
Having calculated all points $x^k$, we then proceed to find a new solution $x'$ that is as close to all points $x^k$ as possible:
\begin{align*}
\min \ & d \\
\text{s.t. } & \sum_{i=1}^n x'_i = 1\\
& \sum_{i=1}^n (x'_i - x^k_i)^2 \le d & \forall k=1,\ldots,N \\
& x' \ge 0
\end{align*}
We then repeat the calculation of closest points, until the change in objective value is sufficiently small. In this setting, the projection method is known to converge to an optimal solution (see, e.g., \cite{dattorro2010convex,marcdiss})

\subsection{Instances and computational setting}

We consider instances with $n=5,10,15,20,25,30$ and $N=5,10,15,20,25,30$, where we generate 100 instances for each setting of $n$ and $N$ (i.e., a total of $6\cdot6\cdot100 = 3600$ instances were generated). An instance is generated by sampling uniformly randomly values for $p^k_i$ in the range $\{1,\ldots,100\}$.

For each instance, we first calculate the two lexicographic solutions with respect to recovery distance and profit. Then the following problems were solved:
\begin{itemize}
\item We solve the classic scalarization, {\Pdelta}, i.e., (Rec)
with bounds on the recovery distance, 
where the bounds are calculated by choosing 50 equidistant points within the relevant region given by the lexicographic solutions. This approach is denoted as Rec-P.

\item For solving the new scalarization, i.e., {\Peps}, we used three different
approaches:

\begin{itemize}
\item Using also 50 equidistant bounds on the profit, we solve recoverable-robust problems {\Peps} directly. This approach is denoted as Rec-D.

\item In the same setting as for Rec-D, we use the iterative projection algorithm instead of solving the recovery problem directly with Cplex. This is denoted as Rec-It.

\item Finally, as preliminary experiments showed that Rec-It is especially fast if the bound on the profit $P$ is large, we used a mixed approach that uses Rec-D for the 2/3 smallest bounds on $P$, and Rec-It for the 1/3 largest bounds on $P$. This is denoted as Rec-M.
\end{itemize}

\end{itemize}

We used Cplex v.12.6 to solve the resulting quadratic programs. The experiments were conducted on a computer with a 16-core Intel Xeon E5-2670 processor, running at 2.60 GHz with 20MB cache, and Ubuntu 12.04. Processes were pinned to one core.

\subsection{Results}

We show the average computation times for the biobjective portfolio problem in Table~\ref{tabresults}.

\begin{table}[htbp]
\centering
\begin{tabular}{r|r|rrrr}
$n$ & $N$ & Rec-P & Rec-D & Rec-It & Rec-M \\
\hline
\multirow{5}{*}{5} &  5  &  {\bf 0.29}  &  0.32  &  1.70  &  0.48 \\
 &  10  &  {\bf 0.48}  &  0.56  &  2.56  &  0.77 \\
 &  15  &  {\bf 0.74}  &  0.91  &  3.43  &  1.16 \\
 &  20  &  {\bf 0.99}  &  1.15  &  3.78  &  1.40 \\
 &  25  &  {\bf 1.26}  &  1.49  &  4.14  &  1.75 \\
 &  30  &  1.55  &  1.86  &  5.30  &  2.18 \\
\hline
\multirow{5}{*}{10}  &  5  &  {\bf 0.57}  &  0.62  &  3.31  &  0.74 \\
  &  10  &  {\bf 1.45}  &  1.53  &  6.22  &  1.67 \\
  &  15  &  2.70  &  {\bf 2.59}  &  8.60  &  2.79 \\
  &  20  &  4.42  &  {\bf 4.11}  &  13.15  &  4.33 \\
  &  25  &  {\bf 3.70}  &  4.12  &  17.95  &  4.99 \\
  &  30  &  {\bf 4.47}  &  5.04  &  21.36  &  6.38 \\
\hline
\multirow{5}{*}{15} &  5  &  {\bf 0.85}  &  0.96  &  5.08  &  1.04 \\
  &  10  &  2.85  &  2.97  &  8.62  &  {\bf 2.84} \\
  &  15  &  5.46  &  5.13  &  14.82  &  {\bf 4.94} \\
  &  20  &  10.85  &  9.16  &  25.65  &  {\bf 8.80} \\
  &  25  &  18.08  &  14.56  &  32.12  &  {\bf 13.31} \\
  &  30  &  {\bf 10.37}  &  20.83  &  46.30  &  19.07 \\
\hline
\multirow{5}{*}{20}   &  5  &  {\bf 1.19}  &  1.25  &  6.74  &  1.33 \\
  &  10  &  4.86  &  5.08  &  13.60  &  {\bf 4.50} \\
  &  15  &  11.23  &  10.03  &  25.10  &  {\bf 8.91} \\
  &  20  &  20.48  &  13.22  &  34.78  &  {\bf 12.27} \\
  &  25  &  30.02  &  22.81  &  49.34  &  {\bf 19.98} \\
  &  30  &  44.38  &  36.88  &  65.80  &  {\bf 31.45} \\
\hline
\multirow{5}{*}{25}   &  5  &  1.57  &  {\bf 1.51}  &  8.08  &  1.59 \\
  &  10  &  5.06  &  {\bf 4.22}  &  19.55  &  4.23 \\
  &  15  &  10.58  &  8.62  &  29.81  &  {\bf 8.35} \\
  &  20  &  19.04  &  15.10  &  46.93  &  {\bf 14.19} \\
  &  25  &  35.82  &  28.18  &  75.60  &  {\bf 26.09} \\
  &  30  &  53.97  &  42.80  &  102.47  &  {\bf 38.49} \\
\hline
\multirow{5}{*}{30}   &  5  &  2.02  &  {\bf 1.83}  &  9.77  &  1.84 \\
  &  10  &  6.27  &  {\bf 4.98}  &  25.59  &  5.16 \\
  &  15  &  13.44  &  {\bf 10.29}  &  45.68  &  10.32 \\
  &  20  &  24.04  &  {\bf 18.31}  &  71.05  &  18.44 \\
  &  25  &  39.49  &  29.53  &  101.90  &  {\bf 28.90} \\
  &  30  &  68.43  &  51.67  &  145.12  &  {\bf 47.77} \\
\end{tabular}
\caption{Average computation times in $s$ to calculate Pareto solutions.}\label{tabresults}
\end{table}

The best average computation time per row is printed in bold. Note that Rec-It requires higher computation times than any other approach; however, in combination with Rec-D (i.e., Rec-M), it is highly competitive. While Rec-P performs well for smaller instances, Rec-D and Rec-M perform best for larger instances.

There are some surprises in Table~\ref{tabresults}, which are not due to outliers. For Rec-P and $n=10$, one can see that solving $N=20$ takes longer than solving $N=25$. The same holds for $n=15$, $N=25$ and $N=30$. Also, for $N=15$, we find that Rec-P is faster for $n=25$ than for $n=20$ (the same holds for Rec-D). This behavior disappears for large $n$ and $N$.

Summarizing, our experimental results show that switching perspective from the classic recoverable-robust approach {\Pdelta} that maximizes the profit subject to some fixed recovery distance to the {\Peps} approach we suggest, in which the distance is minimized subject to some bound on the profit, results in improved computation times. These computation times are further improved by applying methods from location theory, that can allow the {\Peps} version to be solved more efficiently.

\section{Summary and conclusion}
\label{sec-conclusion}

In this paper, we introduced a location-analysis based point of view to the problem of finding recoverable-robust solutions to uncertain optimization problems.
Table~\ref{tab-1} summarizes the results we obtained.
\smallskip

\begin{table}
\begin{adjustwidth}{-2cm}{-2cm}
{\small
\begin{tabular}{p{2cm}|p{2cm}|p{2cm}|p{2cm}|p{2.2cm}|p{4cm}}
uncertainty & constraints     &  uncertainty  & rec. costs  & deterministic  & results \\
set $\cU$   & $F(\cdot, \xi)$ & $F(x, \cdot)$ & $d$         & constraints $\X$ & \\[3pt]
\hline \hline
 & & & & & \\[-4pt]
finite     & quasiconvex       &  arbitrary   & convex & convex and closed &- {\Peps} convex problem (Lemma~\ref{lem-convex}) \\[2mm]
           &                   &              &                       & $\X=\R^n$          &- Reduction to (Rec($\bar{\cU}$) for
						                         smaller sets $\bar{\cU}$ (Theorem~\ref{theo-helly}) \\[2pt]
\hline
& & & & &\\[-4pt]
finite     & linear            &  arbitrary   & block norm & polyhedron &- {\Peps} linear problem (Theorem~\ref{theo-finite}) \\[2pt]
\hline
& \multicolumn{2}{p{3cm}|}{} & & &\\[-4pt]
& \multicolumn{2}{p{3cm}|}{} & & & \\[-4pt]
polyhedron & \multicolumn{2}{c|}{jointly quasiconvex}& convex & closed & - Pareto solution w.r.t. extreme points
of $\cU$ is Pareto (Theorem~\ref{th:quasiconvex})\\[2pt]
\hline
& & & & &\\[-4pt]
polyhedron & convex            &  quasiconvex, right-hand & convex & closed & - solution w.r.t extreme points
 of $\cU$ is Pareto (Corollary~\ref{theo-right-hand}) \\[2mm]
                  &            &         side    &      & convex and closed&  - {\Peps} convex problem (Corollary~\ref{cor-right-hand})\\[2pt]
\hline
& & & & &\\[-4pt]
polyhedron &  linear            & quasiconvex, right-hand side& block norm & polyhedron &  - {\Peps} linear problem
 (Theorem~\ref{theo-together}) \\[2pt]

\hline \hline
\end{tabular}
}
\end{adjustwidth}
\caption{\label{tab-1} Summary of properties of (Rec) and {\Peps} depending on the optimization problem
$\rm P(\xi)$, the uncertainty set $\cU$, the type of uncertainty, and the recovery costs.}
\end{table}

The following variation of (Rec) should be mentioned:
In many cases it might not be appropriate to just look at the worst-case objective
function of the recovered solutions, because there might be one very bad scenario which is
the only relevant one. Pareto efficient solutions would hence neglect the objective function 
values of all other scenarios.

This might lead to another goal, namely to be as close as possible to an optimal
solution in \emph{all} scenarios instead of only looking at a few scenarios which will be
very bad anyway. This leads to the following problem in which we bound the difference
between the objective value of the recovered solution and the best possible objective
function value in the worst case:

\begin{eqnarray*}
(\widehat{{\rm Rec}}) \hspace{5mm} & {\rm minimize} \ 
\left( {\bf \hat{f}}(y), {\bf r}(x,y) \right)=
\left( \sup_{\xi \in \cU} f(y(\xi),\xi)-f^*(\xi), \sup_{\xi \in \cU} d(x,y(\xi)) \right)\\
& \mbox{s.t. } F(y(\xi),\xi) \leq 0\ \ \mbox{ for all } \xi \in \cU\\
& x \in \X, y:\cU \to \X
\end{eqnarray*}

The new objective function ${\bf \hat{f}}$ in ($\widehat{{\rm Rec}}$) can be interpreted 
as a minmax-regret approach as described in \cite{KouYu97}. 
Again, we can look at the scalarizations of this problem. Instead of \Peps we receive

\begin{eqnarray*}
(\widehat{{\rm Rec}}(\varepsilon))
\hspace{1cm} & {\rm minimize} \ \sup_{\xi \in \cU} d(x,y(\xi))\\
\mbox{s.t. } 
& f(y(\xi),\xi) - f^*(\xi) \leq \varepsilon\ \ \mbox{ for all } \xi \in \cU\\
& F(y(\xi),\xi) \leq 0\ \ \mbox{ for all } \xi \in \cU\\
&x \in \X, y:\cU \to \X
\end{eqnarray*}
In case that $f^*(\xi)$ is known for all $\xi \in \cU$, ($\widehat{{\rm Rec}}(\varepsilon)$)
admits similar properties as {\Peps}.  
\medskip

Note that the lexicographic solution of ($\widehat{{\rm Rec}}(\varepsilon)$) with respect to
(${\bf \hat{f}},{\bf r}$) requires to find optimal solutions for each scenario $\xi \in \cU$ 
which can be reached with minimal recovery costs. It can be found by solving
$(\widehat{{\rm Rec}}(0))$. This is exactly the robustness
approach \emph{recovery-to-optimality} which has been described in \cite{Goerigk20141},
see \cite{Atmos2010-goerigk,Tapas2011-Goerigk-Sch} for scenario-based approaches
for its solution. On the other hand, the lexicographic solution of ($\widehat{{\rm Rec}}(\varepsilon)$) 
with respect to (${\bf r}, {\bf \hat{f}}$) is related to minmax regret robustness.

\medskip

Ongoing research includes the analysis of other special cases of (Rec) as well as its application 
to other types of problems e.g. from traffic
planning or evacuation. We also work on generalizations to multi-objective uncertain optimization
problems as already done for several minmax robustness concepts \cite{EIS13}.

\newcommand{\etalchar}[1]{$^{#1}$}

\end{document}